\documentclass[11pt,reqno]{amsart}
\usepackage{amsmath, amssymb, amsthm}

\renewcommand{\(}{\left\(}
\renewcommand{\)}{\right\)}
\renewcommand{\[}{\left\[}
\renewcommand{\]}{\right\]}

\numberwithin{equation}{section}
 \theoremstyle{plain}
\newtheorem{theorem}{Theorem}[section]

\newtheorem{conjecture}[theorem]{Conjecture}

\newtheorem{problem}[]{Problem}

   \makeatletter
\def\proof{\@ifnextchar[{\@oproof}{\@nproof}}
\def\@oproof[#1][#2]{\trivlist\item[\hskip\labelsep\textit{#2 Proof of\
#1.}~]\ignorespaces}
\def\@nproof{\trivlist\item[\hskip\labelsep\textit{Proof.}~]\ignorespaces}

\makeatother
\begin{document}
\title[Elementary symmetric polynomials and a potentially injective family of maps on partitions]{Elementary symmetric polynomials and a potentially injective family of maps on partitions}
\author{Aman Devnani}
\address{Aman Devnani, Birla Institute of Technology \& Science Pilani, Vidya vihar, Pilani, Rajasthan - 333031, India.}
\email{f20231026@pilani.bits-pilani.ac.in}


\author{Pramod Eyyunni}
\address{Pramod Eyyunni, Department of Mathematics, Birla Institute of Technology \& Science Pilani, Vidya vihar, Pilani, Rajasthan - 333031, India.} 
\email{pramod.eyyunni@pilani.bits-pilani.ac.in}

\thanks{$2020$ \textit{Mathematics Subject Classification.} Primary 11P81, 05A17; Secondary 05A20. }

\begin{abstract}
In this article, we provide an infinite family of examples to disprove a recent conjecture due to Ballantine and her collaborators on the injectivity of a class of maps, namely $pre_k$, defined on integer partitions. These maps arise from applying the sequence of elementary symmetric polynomials to integer partitions, where $pre_k$ is associated with the $k$th polynomial. Subsequently, we state a modified version of their conjecture. Throwing fresh light on these class of maps, we study the inter-relationships between them, deviating from the approaches so far, which study these maps one at a time. Though one case of the conjecture ($k=2$) has now been settled independently by the work of Ballantine and collaborators, and Li, we provide alternate proofs of three subcases corresponding to this settled case. We also discuss lower bounds for the number of partitions of $n$ which are in the image of the map $pre_2$.  
\end{abstract}  
\maketitle
\section{Introduction}
The study of the action of elementary symmetric polynomials on integer partitions is a topic of recent interest. Given the $k$th elementary symmetric polynomial, namely,
\begin{equation}
e_k(X_1, X_2, \dots, X_n) =
\begin{cases}
 \displaystyle\sum_{1\leq i_1 < i_2 < \cdots < i_k \leq n} X_{i_1}X_{i_2} \cdots X_{i_k}, \quad \textup{if} \ k \leq n, \\
0, \quad \textup{if} \ k > n,
\end{cases}
\end{equation}
Ballantine, Beck and Merca \cite{BBM24} defined, for $\lambda = (\lambda_1, \lambda_2, \dots, \lambda_{\ell})$, $pre_k(\lambda)$ to be the partition with the multiset of parts given by
\begin{equation*}
\{ \lambda_{i_1} \lambda_{i_2} \cdots \lambda_{i_k} \ | \ 1 \leq i_1 < i_2 < \cdots < i_k \leq \ell\}, \quad \textup{if} \ \ell \geq k,
\end{equation*}
and to be the empty partition $\emptyset$ if $\ell < k$.  For example, $pre_2 (7, 4, 4) = (7 \times 4, \ 7\times 4, \ 4\times 4) = (28, 28, 16)$. The special case $\ell = k$ has been studied under the name `norm' of the partition $\lambda$ \cite{norm}. Note also that $pre_1(\lambda) = \lambda$ itself. If we let $\mathcal{P}(n)$ to denote the set of partitions of $n$, we have that $pre_1$ is injective on $\mathcal{P}(n)$ since it is the identity map. In \cite[Conjectures 12, 13]{BBM24}, the following conjectures were proposed:
\begin{conjecture}\label{conj12}
For $n \geq 0$, the map $pre_2$ is injective on $\mathcal{P}(n)$.
\end{conjecture}
\begin{conjecture}\label{conj13}
For $k \geq 3$, $pre_k$ is injective on the subset of $\mathcal{P}(n)$ consisting of partitions with at least $k$ parts.
\end{conjecture}
We first disprove Conjecture \ref{conj13} in the form given above. This is owing to the failure of the injectivity of $pre_k$ on partitions of $n$ with $k$ parts, for infinitely many $n$. In Theorem \ref{notinjk} below, we give an infinite family of pairs of partitions of the same numbers with $k$ parts which have the same image under $pre_k$, for each $k \geq 3$.
\begin{theorem}\label{notinjk}
Conjecture \ref{conj13} is false. More precisely, $pre_k$ is not injective on partitions of $n$ with $k$ parts, for infinitely many $n$.
\end{theorem}
But this does not render Conjecture \ref{conj13} entirely futile because of the following reason: If a partition $\lambda$ has $\ell$ parts, then note that for $k \leq \ell$, $pre_k(\lambda)$ has $\binom{\ell}{k}$ parts. Therefore, if $\lambda$ and $\mu$ have $\ell_1$ and $\ell_2$ parts respectively such that $pre_k(\lambda) = pre_k(\mu)$, then the number of parts of both the image partitions must be the same, i.e., $\binom{\ell_1}{k} = \binom{\ell_2}{k}$. This implies that $\ell_1 = \ell_2$ as $\binom{n}{k}$ is an increasing function of $n$ for fixed $k$. That means partitions with different numbers of parts (but both of them at least $k$) will necessarily have different images under the map $pre_k$. Consequently, we modify Conjecture \ref{conj13} as follows.
\begin{conjecture} \label{conj13new}
For $k \geq 3$, $pre_k$ is injective on partitions of $n$ with $j$ parts for each $j \geq k+1$.
\end{conjecture}
As a result of this rephrasal, we observe that it may be fruitful to deal with the family of maps $pre_k, \ k \geq 1$ as a collective unit rather than trying to prove that each one of them is injective separately. In this direction, we prove Theorem \ref{prekprel-k}. 
\begin{theorem}\label{prekprel-k}
 For some $k$, $1 \leq k < \ell$, if $pre_k$ is injective on partitions of $n$ with $\ell$ parts, then $pre_{\ell - k}$ is also injective on this same set of partitions.
\end{theorem}
For instance, since $pre_2$ is injective on partitions of $n$ with five and six parts (see Theorem \ref{456} below), we deduce that $pre_3$ is injective on partitions of $n$ with five parts and $pre_4$ is injective on partitions of $n$ with six parts. Also, for $k \geq 2$, $pre_k$ is injective on partitions of $n$ with $k+1$ parts, since $pre_1$, being the identity map on partitions, is injective on partitions of $n$ with $k+1$ parts.

Apart from framing Conjecture \ref{conj12} in \cite{BBM24}, Ballantine et al. proved the truth of it for partitions having at most three parts, using a `lattice of parts of $pre_2(\lambda)$' approach \cite[Proposition 1]{BBM24}. This simply means that the partition $\lambda$ gives rise to a partial order on the parts of $pre_2(\lambda)$. They also requested a proof along these lines for partitions with more than three parts. We provide such a proof for partitions having at most six parts in Theorem \ref{456}. Subsequently, Ballantine et al. \cite{BNTWZ} and Li \cite{SJL} independently proved the whole of Conjecture \ref{conj12} via alternate methods.
\begin{theorem}\label{456}
The map $pre_2$ is injective on partitions of $n$ with four, five, and six parts.
\end{theorem} 
We next consider the question: How many partitions of $n$ are in the image of the map $pre_2$? The answer is at least one for each $n$, since the partition $n = pre_2(\lambda)$, for $\lambda = n + 1$, falls under this class. Let $Pre_2(n)$ be the set of partitions of $n$ which are in the image of $pre_2$. Set $pre_2(n) := |Pre_2(n)|$. From our earlier observation, we have that $pre_2(n) \geq 1$ for all $n \geq 1$. It is interesting to see when there are other such partitions for a given $n$. We have the following result in this regard.
\begin{theorem}\label{pre2nsize}
For $n \geq 1$, we have that
$$
pre_2(n) \geq 
\begin{cases}
\frac{\tau(n+1)}{2}, \quad \text{if $n+1$ is not a perfect square}, \\
\frac{\tau(n+1) + 1}{2}, \quad \text{if $n+1$ is a perfect square},
\end{cases}
$$
where $\tau(n)$ counts the number of positive divisors of $n$.
\end{theorem}
 
In particular, if $n+1$ is composite and not a perfect square, then it has at least four divisors and consequently, $pre_2(n) \geq 2$. For instance, if $n = 23$, then the following partitions of $n$ are in the image of $pre_2$, namely, $11 + 11 + 1 \ (= pre_2(11+1+1)), \ 14+7+2 \ (= pre_2(7+2+1)), \ 15 + 5 + 3 \ (= pre_2(5+3+1))$. Thus, $pre_2(23) \geq 3$.

This paper is organized as follows. In Section \ref{disprove}, we prove Theorems \ref{notinjk} and \ref{prekprel-k}. Next, in Section \ref{456proof}, we do the proof of Theorem \ref{456}, on the injectivity of $pre_2$ on the set of partitions of $n$ with four, five and six parts. Moving on, we initiate the study of $pre_2(n)$ in Section \ref{pre2n} by deriving the proof of Theorem \ref{pre2nsize}. Finally, we mull over possible directions to explore via some problems in the concluding section. 
 
\section{Disproving Conjecture \ref{conj13} and linking pairs of $pre_k$ maps} \label{disprove}
We first exhibit an infinite three parameter family of pairs of partitions that disproves Conjecture \ref{conj13}.
\begin{proof}[Theorem \ref{notinjk}][]
If $k \geq 3$, let $\alpha_k = (6, 6, \underbrace{1, \cdots, 1}_{k-2}), \ \beta_k = (9, 2, 2, \underbrace{1, \cdots, 1}_{k-3})$. Note that both are partitions of $k+10$ having $k$ parts each. We have $pre_k(\alpha_k) = pre_k(\beta_k) = 36$. More specifically, for $k=3$, this gives us the partitions $(6, 6, 1), \ (9, 2, 2)$, of 13 having the same image, 36, under $pre_3$. This, in fact, readily generates an infinite family of pairs of partitions with identical $pre_3$ images, namely, $(6m, 6m, m), \ (9m, 2m, 2m)$, for each $m \geq 1$. Each is a partition of $13m$, with the $pre_3$ image being $36m^3$. We now give another family where the parts of each of the partitions are relatively prime, namely, $(3m, 2m-1, 2), \ (4m-2, m, 3), \ m \geq 3$. For $(3m, 2m-1, 2)$, clearly, 2 and $2m-1$, being odd, are relatively prime and hence the triplet of numbers has this property as well. And for $(4m-2, m, 3)$, if 3 does not divide $m$, then both of them are relatively prime. If not, then $4m-2$ is not divisible by 3 and hence is relatively prime to 3. In either case, the triplet of numbers has gcd equal to one. Moreover, observe that both are partitions of $5m+1$ having $3$ parts, with the identical $pre_3$ image $6m(2m-1)$. 

More generally, if $p > q$ are two primes, then the family of pairs of partitions given by $(q(1 + m(p-1)), \ 1 + m(q-1), \ p), \ (p(1 + m(q-1)), \ 1 + m(p-1), \ q)$, where $m \geq \frac{p-1}{q-1}$, are both partitions of $p+q+1+m(pq-1)$ having three parts with the same $pre_3$ image $pq(1 + m(p-1))(1 + m(q-1))$. The condition on $m$ is necessary to ensure that in the first named partition, we have $1 + m(q-1) \geq p$. The second family alluded to in the previous paragraph, that is, $(3m, 2m-1, 2), \ (4m-2, m, 3)$ can be obtained from the special case $p=3, \ q=2$. Moreover, for each $k \geq 3$, we can extend the pair at the beginning of this paragraph to $(q(1 + m(p-1)), \ 1 + m(q-1), \ p, \ \underbrace{1, \cdots, 1}_{k-3}), \ (p(1 + m(q-1)), \ 1 + m(p-1), \ q, \ \underbrace{1, \cdots, 1}_{k-3})$, with $m \geq \frac{p-1}{q-1}$, giving us an infinite three parameter family of pairs of partitions of $p+q+k+m(pq-1)-2$ with $k$ parts, having the same image, $pq(1 + m(p-1))(1 + m(q-1))$, under $pre_k$. That means, for infinitely many $n$, $pre_k$ fails to be injective on partitions of $n$ with $k$ parts.
\end{proof}

We now link pairs of maps of the type $pre_k$.
\begin{proof}[Theorem \ref{prekprel-k}][]
Let $\lambda = (\lambda_1, \cdots, \lambda_{\ell}), \ \mu = (\mu_1, \cdots, \mu_{\ell})$ with $pre_{\ell - k}(\lambda) = pre_{\ell - k}(\mu)$. Thus, the product of parts of the partitions $pre_{\ell - k}(\lambda), \ pre_{\ell - k}(\mu)$ are also the same. Since each $\lambda_i$ occurs once in exactly $\binom{\ell - 1}{\ell - k - 1}$ parts of $pre_{\ell - k}(\lambda)$, the product of parts of $pre_{\ell - k}(\lambda)$ is therefore $(\prod_{i=1}^{\ell} \lambda_i)^{\binom{\ell - 1}{\ell - k - 1}}$, and that of $pre_{\ell - k}(\mu)$ is, similarly, $(\prod_{i=1}^{\ell} \mu_i)^{\binom{\ell - 1}{\ell - k - 1}}$. Thus, $(\prod_{i=1}^{\ell} \lambda_i)^{\binom{\ell - 1}{\ell - k - 1}} = (\prod_{i=1}^{\ell} \mu_i)^{\binom{\ell - 1}{\ell - k - 1}}$, giving us $\prod_{i=1}^{\ell} \lambda_i = \prod_{i=1}^{\ell} \mu_i$. The fact that $pre_{\ell - k}(\lambda) = pre_{\ell - k}(\mu)$ means that every product of $\ell - k$ numbers from the partition $\lambda$ equals a similar kind of product of the parts of $\mu$, combined with $\prod_{i=1}^{\ell} \lambda_i = \prod_{i=1}^{\ell} \mu_i$, implies that there is a correspondence between the product of any $k$ parts of $\lambda$ with that of $\mu$. This thereby implies that $pre_k(\lambda) = pre_k(\mu)$. Since $pre_k$ is injective, this finally tells us that $\lambda = \mu$. Hence, $pre_{\ell-k}$ is also injective.
\end{proof}
\section{A proof of injectivity of $pre_2$ on partitions of $n$ with four, five and six parts}\label{456proof}
As mentioned in the introduction, we prove Theorem \ref{456} in the spirit of the `lattice of parts' approach of Ballantine, Beck and Merca \cite[Proposition 1]{BBM24}.
\begin{proof}[Theorem \ref{456}][]
Let $\lambda = (\lambda_1, \lambda_2, \lambda_3, \lambda_4)$, and $\mu = (\mu_1, \mu_2, \mu_3, \mu_4)$ be partitions of $n$ with $pre_2(\lambda) = pre_2(\mu) = \nu$. Then $\nu = (\lambda_1 \lambda_2, \lambda_1 \lambda_3, \cdots, \lambda_2 \lambda_4, \lambda_3 \lambda_4) = (\mu_1 \mu_2, \mu_1 \mu_3, \cdots, \mu_2 \mu_4, \mu_3 \mu_4)$. Therefore,
\begin{equation} \label{eqn1}
\lambda_1 \lambda_2 = \mu_1 \mu_2, \quad \lambda_1 \lambda_3 = \mu_1 \mu_3, \quad
\lambda_2 \lambda_4 = \mu_2 \mu_4, \quad \lambda_3 \lambda_4 = \mu_3 \mu_4. 
\end{equation}
In other words,
\begin{equation}\label{eqn2}
\frac{\lambda_1}{\mu_1} = \frac{\mu_2}{\lambda_2} = \frac{\mu_3}{\lambda_3} = \frac{\lambda_4}{\mu_4} = k \quad (\text{say, a positive rational number})
\end{equation} 
Since $\lambda_1 + \lambda_2 + \lambda_3 + \lambda_4 = n = \mu_1 + \mu_2 + \mu_3 + \mu_4$, we have from \eqref{eqn2}, $k \mu_1 + \mu_2/k + \mu_3/k + k\mu_4 = \mu_1 + \mu_2+ \mu_3 + \mu_4$. This implies that 
\begin{equation}\label{eqn3}
\mu_1 + \mu_4 = \frac{\mu_2 + \mu_3}{k}.
\end{equation}
Note that the ordering among the four indicated parts of $\nu$ is fixed, as shown in \eqref{eqn1}, but the remaining two parts are unrestricted, order-wise. We now consider two cases depending on two possible equalities.  \\
\textbf{Case 1:} $\lambda_1 \lambda_4 = \mu_1 \mu_4$, and $\lambda_2 \lambda_3 = \mu_2 \mu_3$. Then $k = \frac{\lambda_1}{\mu_1} = \frac{\mu_4}{\lambda_4}$. But from \eqref{eqn2}, we know that $\frac{\mu_4}{\lambda_4} = 1/k$. Thus, $k = 1/k$, implying $k=1$ (since $k$ is positive). Again, \eqref{eqn2} implies, for $k=1$, that $\lambda_i = \mu_i$ for all $i=1,2,3,4$, which offcourse means that $\lambda = \mu$. \\
\textbf{Case 2:} $\lambda_1 \lambda_4 = \mu_2 \mu_3$, and $\lambda_2 \lambda_3 = \mu_1 \mu_4$. Then $k \mu_1 \times k \mu_4 = \mu_2 \mu_3$, implying that 
\begin{equation}\label{eqn4}
\mu_1 \mu_4 = \frac{\mu_2 \mu_3}{k^2} = \frac{\mu_2}{k}\frac{\mu_3}{k}.
\end{equation} 
But from \eqref{eqn3}, \eqref{eqn4}, we get that $\mu_1, \mu_4$ are the roots of the quadratic equation $x^2 - (\frac{\mu_2}{k} + \frac{\mu_3}{k})x + \frac{\mu_2}{k}\frac{\mu_3}{k} = 0$. But the roots of this equation are clearly $\frac{\mu_2}{k}, \frac{\mu_3}{k}$ and $\frac{\mu_2}{k} \geq \frac{\mu_3}{k}$. Since $\mu_1 \geq \mu_4$, we deduce that $\mu_1 = \frac{\mu_2}{k}$ and $\mu_4 = \frac{\mu_3}{k}$. Moreover, $\mu_1 = \frac{\mu_2}{k} \leq \frac{\mu_1}{k}$ , so $k \leq 1$. Similarly, $\mu_4 = \frac{\mu_3}{k} \geq \frac{\mu_4}{k}$, so that $k \geq 1$. We thus conclude that $k=1$, which implies that $\lambda_i = \mu_i$ for all $i$ from \eqref{eqn2}. Hence $\lambda = \mu$. Thus, $pre_2$ is injective on partitions of $n$ with four parts.

We move on to the proof for partitions with five parts. Let $\lambda = (\lambda_1, \lambda_2, \lambda_3, \lambda_4, \lambda_5), \mu = (\mu_1, \mu_2, \mu_3, \mu_4, \mu_5)$ such that $pre_2(\lambda) = pre_2(\mu) = \nu$. From the ordering relations among the parts of $\nu$ as in the proof of partitions with four parts above, we have 
\begin{equation} \label{order relns 5}
\lambda_1 \lambda_2 = \mu_1 \mu_2, \ \lambda_1 \lambda_3 = \mu_1 \mu_3, \ \lambda_3 \lambda_5 = \mu_3 \mu_5, \ \lambda_4 \lambda_5 = \mu_4 \mu_5.
\end{equation}
We note that the product of parts of $pre_2(\lambda)$ and $pre_2(\mu)$ are equal since the partitions are themselves equal, giving us $(\lambda_1 \times \cdots \times \lambda_5)^4 = (\mu_1 \times \cdots \times \mu_5)^4$. Thus, $\prod_{i=1}^{5} \lambda_i = \prod_{i=1}^{5} \mu_i$. But  $\lambda_1 \lambda_2 = \mu_1 \mu_2$ and $\lambda_4 \lambda_5 = \mu_4 \mu_5$, which implies that $\lambda_3 = \mu_3$. From each of the relations in \eqref{order relns 5}, we can then deduce that $\lambda_i = \mu_i$ for the other values of $i$ as well. Therefore, $\lambda = \mu$ and $pre_2$ is injective on partitions of $n$ with five parts.

Coming to partitions having six parts finally, let $\lambda = (\lambda_1, \cdots, \lambda_6), \ \mu = (\mu_1, \cdots, \mu_6)$ with $pre_2(\lambda) = pre_2(\mu) = \nu$. As before, we have the ordering relations 
\begin{equation}\label{order relns 6}
\lambda_1 \lambda_2 = \mu_1 \mu_2, \ \lambda_1 \lambda_3 = \mu_1 \mu_3, \ \lambda_4 \lambda_6 = \mu_4 \mu_6, \ \lambda_5 \lambda_6 = \mu_5 \mu_6.
\end{equation}
We also have the product of parts of $\nu$ equal to $(\prod_{i=1}^{5}\lambda_i)^5 = (\prod_{i=1}^{5} \mu_i)^5$. This implies that $\prod_{i=1}^{5}\lambda_i = \prod_{i=1}^{5} \mu_i$. From \eqref{order relns 6}, we see that $\lambda_1 \lambda_2 \lambda_5 \lambda_6 = \mu_1 \mu_2 \mu_5 \mu_6$. Now, $\prod_{i=1}^{5}\lambda_i = \prod_{i=1}^{5} \mu_i$ implies that $\lambda_3 \lambda_4 = \mu_3 \mu_4$. In a similar vein, we get other order relations as listed below: (we include $\lambda_3 \lambda_4 = \mu_3 \mu_4$ too for convenience)
\begin{equation}\label{order relns II}
\lambda_3 \lambda_4 = \mu_3 \mu_4, \ \lambda_3 \lambda_5 = \mu_3 \mu_5, \ \lambda_2 \lambda_5 = \mu_2 \mu_5, \ \lambda_2 \lambda_4 = \mu_2 \mu_4.
\end{equation}
Therefore, using the relations in \eqref{order relns 6} and \eqref{order relns II}, we get the following equality of ratios:
\begin{equation}\label{ratios}
\frac{\lambda_1}{\mu_1} = \frac{\mu_2}{\lambda_2} = \frac{\mu_3}{\lambda_3} = \frac{\lambda_4}{\mu_4} = \frac{\lambda_5}{\mu_5} = \frac{\mu_6}{\lambda_6}.
\end{equation}
From \eqref{ratios}, we also get the additional relation $\lambda_1 \lambda_6 = \mu_1 \mu_6$. Collecting all the known relations so far, we have
\begin{equation}\label{all relations 1}
\lambda_1 \lambda_2 = \mu_1 \mu_2, \ \lambda_1 \lambda_3 = \mu_1 \mu_3, \ \lambda_1 \lambda_6 = \mu_1 \mu_6, \ \lambda_2 \lambda_4 = \mu_2 \mu_4, \ \lambda_2 \lambda_5 = \mu_2 \mu_5.
\end{equation}
\begin{equation}\label{all relations 2}
\lambda_3 \lambda_4 = \mu_3 \mu_4, \ \lambda_3 \lambda_5 = \mu_3 \mu_5, \ \lambda_4 \lambda_6 = \mu_4 \mu_6, \ \lambda_5 \lambda_6 = \mu_5 \mu_6.
\end{equation}
Out of the remaining parts of $\nu = pre_2(\lambda)$, the largest can either be $\lambda_1 \lambda_4$ or $\lambda_2 \lambda_3$. We have a similar situation with $pre_2(\mu)$ as well. So we will either have $\lambda_1 \lambda_4 = \mu_1 \mu_4, \ \lambda_2 \lambda_3 = \mu_2 \mu_3$ or $\lambda_1 \lambda_4 = \mu_2 \mu_3, \ \lambda_2 \lambda_3 = \mu_1 \mu_4$.

\textbf{Case 1:} $\lambda_1 \lambda_4 = \mu_1 \mu_4, \ \lambda_2 \lambda_3 = \mu_2 \mu_3$. This implies that $\frac{\lambda_2}{\mu_2} = \frac{\mu_3}{\lambda_3} = \frac{\mu_2}{\lambda_2}$, the latter equality following from \eqref{ratios}. But this gives $\lambda_2^2 = \mu_2^2$, and so $\lambda_2 = \mu_2$ as both are positive numbers. From \eqref{ratios}, we then conclude that $\lambda = \mu$.

\textbf{Case 2:} $\lambda_1 \lambda_4 = \mu_2 \mu_3, \ \lambda_2 \lambda_3 = \mu_1 \mu_4$. Note that, after considering all the order relations in \eqref{all relations 1} and \eqref{all relations 2}, the largest parts among the remaining parts of $pre_2(\lambda)$ and $pre_2(\mu)$ are $\lambda_1 \lambda_5$ and $\mu_1 \mu_5$ respectively. Therefore, we deduce that $\lambda_1 \lambda_5 = \mu_1 \mu_5$. Thus, $\frac{\lambda_1}{\mu_1} = \frac{\mu_5}{\lambda_5} = \frac{\lambda_5}{\mu_5}$, the last equality coming from \eqref{ratios}. So, with $\lambda_5^2 = \mu_5^2$, we deduce again that $\lambda_5 = \mu_5$ and hence $\lambda = \mu$.
\end{proof}

\section{Partitions which are in the image of $pre_2$}\label{pre2n}
We next give lower bounds for the function $pre_2(n)$, the number of partitions of $n$ which are in the image of $pre_2$.
\begin{proof}[Theorem \ref{pre2nsize}][]
\textbf{Case 1:} Let $n+1$ be a composite number that is not a perfect square. Then for each factorization $n+1 = (a+1)(b+1)$, with $a, b \geq 1$, we obtain $n = ab + a +b = ab + b + a = pre_2(a + b +1)$ or $pre_2(b + a +1)$. That is, for each factorization of $n+1$ into two numbers greater than $1$, we get a partition for $n$ that is in the image of $pre_2$. Moreover, interpretating $n+1 = (n+1)(0+1)$ with $b = 0$, gives us $n = n = pre_2(n+1)$. Thus, every pair of divisors of $n+1$ gives a partition of $n$ that is in the image of $pre_2$. Therefore, $pre_2(n) \geq \frac{\tau(n+1)}{2}$. Now, if $n+1$ is a prime, then $\frac{\tau(n+1)}{2} = 1$, and hence we again have $pre_2(n) \geq \frac{\tau(n+1)}{2}$.

\textbf{Case 2:} If $n+1$ is a perfect square, then in addition to its factorizations considered in Case 1, we also have $n+1 = (a+1)^2$ for some $a \geq 1$. This gives us $n = a^2 + a +a = pre_2(a+a+1)$. Thus, using Case 1, we then have $pre_2(n) \geq \frac{\tau(n+1) - 1}{2} + 1 = \frac{\tau(n+1) + 1}{2}$. 
\end{proof}


\section{Concluding Remarks}

We showed in Theorem \ref{notinjk} that for $k \geq 3$, there are infinitely many $n$ for which $pre_k$ is not injective on partitions of $n$ with $k$ parts. But do there exist $n$ for which the injectivity still holds?
\begin{problem}
Are there only finitely many $n$ such that $pre_3$ is injective on partitions of $n$ with three parts? In particular, do we have non-injectivity for all $n > 18$, as the observed data seems to suggest? What about this phenomenon for $pre_k$ for $k \geq 4$? For sufficiently large $n$ (depending on $k$), do we have that $pre_k$ is not injective on partitions of $n$ with $k$ parts?
\end{problem}
It would be interesting to see if the proof of Theorem \ref{456} can be extended to partitions with more than six parts.
\begin{problem}
Can a proof by lattice analysis be done for the injectivity of $pre_2$ on partitions of $n$ with more than six parts?
\end{problem}
In the course of the proof of Case 1 of Theorem \ref{pre2nsize}, we saw that if $n+1$ is a prime, then $pre_2(n) \geq 1$. We thus ask
\begin{problem}
Does there exist an $n \geq 1$ such that $pre_2(n) = 1$?
\end{problem}


%
%
%

\end{document}